\documentclass{amsart}

\usepackage{hyperref}

\usepackage{esint}
\usepackage{amssymb}
\usepackage{tikz}
\usepackage{graphicx}
\usepackage{amsrefs,enumitem}

\newtheorem{theorem}{Theorem}
\newtheorem{lemma}[theorem]{Lemma}

\newtheorem{proposition}[theorem]{Proposition}
\theoremstyle{definition}

\newtheorem{assumption}[theorem]{Assumption}

\newcommand{\eref}[1]{(\ref{e.#1})}
\newcommand{\tref}[1]{Theorem \ref{t.#1}}
\newcommand{\lref}[1]{Lemma \ref{l.#1}}
\newcommand{\pref}[1]{Proposition \ref{p.#1}}
\newcommand{\cref}[1]{Corollary \ref{c.#1}}

\newcommand{\partref}[1]{\ref{part.#1}}

\newcommand{\aref}[1]{Appendix \ref{a.#1}}
\newcommand{\asref}[1]{Assumption \ref{as.#1}}

\numberwithin{theorem}{section}
\numberwithin{equation}{section}

\newcommand{\R}{\mathbb{R}}

\newcommand{\E}{\mathbb{E}}

\renewcommand{\P}{\mathbb{P}}
\newcommand{\grad}{\nabla}
\newcommand{\diam}{\textup{diam}}

\def\XXint#1#2#3{{\setbox0=\hbox{$#1{#2#3}{\int}$ }
\vcenter{\hbox{$#2#3$ }}\kern-.6\wd0}}

\newcommand{\ep}{\varepsilon}

\newcommand{\tr}{\textup{tr}}

\begin{document}
\title{Mean Curvature Flow with Positive Random Forcing in 2-D}
\author{William M Feldman}
\address{Institute for Advanced Study, 1 Einstein Dr, Princeton, NJ 08540}
\email{wfeldman@math.ias.edu}
\maketitle
\begin{abstract}
We consider the forced mean curvature flow in $2$-d, finite range of dependence and positive random forcing. We prove flatness and existence of effective speed for initially flat propagating fronts.  This is the analogue, in random media, of a result of Caffarelli and Monneau \cite{CaffarelliMonneau}.  The main new tools are a large scale Lipschitz estimate for the arrival time function, and a quantitative uniqueness result which does not use uniform local regularity. 
\end{abstract}
\section{Introduction}
Consider the interface evolution equation
\begin{equation}\label{e.level set}
 u_t =\tr\left[(I - \tfrac{Du \otimes Du}{|Du|^2})D^2u\right] + c(x)|Du|   \ \hbox{ in } \R^2 \times (0,\infty).
\end{equation}
 The forcing term $c >0$ is a random field satisfying a finite range of dependence assumption.  The PDE \eref{level set} is the level set form of the front propagation problem,
 \begin{equation}\label{e.front}
V_n = -\kappa + c(x) 
\end{equation}
 where $V_n$ is the outward normal velocity of an evolving (oriented) hypersurface $\Gamma_t = \partial S_t$ and $\kappa$ is the mean curvature (positive for outward oriented boundaries of convex regions).  We study the long time behavior of the interface $\Gamma_t$ when $S_0$ is a half-space $\{ x \cdot e \leq 0\}$.  We will show that initially flat fronts stay approximately flat and propagate with an effective speed $\bar{c}(e)$.
 \begin{equation}\label{e.gammahom}
  \frac{1}{t} \Gamma_t \to \{ x \cdot e = \bar{c}(e) \} \ \hbox{ as } \ t \to \infty.
  \end{equation}

 Under hyperbolic space-time rescaling one can see that the study of the long-time behavior of \eref{level set} is closely related to the asymptotic behavior as $\ep \to 0$ of the problem,
 \begin{equation}\label{e.level set hom ep}
 u^\ep_t = \ep\tr\left[(I - \tfrac{Du^\ep \otimes Du^\ep}{|Du^\ep|^2})D^2u^\ep\right] + c(\tfrac{x}{\ep})|Du^\ep| \ \hbox{ in } \R^2 \times (0,\infty)
\end{equation}
with initial data $u(x,0) = u_0(x)$.  In this context one could hope to show that as $\ep \to 0$ the solutions $u^\ep$ converge to $\bar{u}$ the solutions of the ``homogenized problem",
\begin{equation}\label{e.level set hom}
  \bar{u}_t = \bar{c}(\tfrac{D\bar{u}}{|D\bar{u}|})|D\bar{u}|  \ \hbox{ in } \R^2 \times (0,\infty) \ \hbox{ with } \ \bar{u}(x,0) = u_0(x).
 \end{equation}
Here the asymptotic speeds $\bar{c}(e)$ appear again, now as the level set velocity for the homogenized problem.

 This problem was considered in periodic media by Lions and Souganidis \cite{LionsSouganidis} who showed the existence of asymptotic front speeds in $d \geq 2$ under the following condition
 \begin{equation}
 \inf_{x \in \R^d} (c(x)^2-(d-1)|Dc(x)|) >0.
 \end{equation}
 This condition, which we call the Lions-Souganidis condition, as in \cite{ArmstrongCardaliaguet}, is necessary and sufficient to obtain uniform Lipschitz estimates on the solutions of the approximate corrector problem associated with \eref{level set}.  However it was not clear whether the Lions-Souganidis condition was necessary for homogenization.  Caffarelli and Monneau subsequently proved in \cite{CaffarelliMonneau} that, in $d=2$, it is sufficient that $c>0$, while, in $d\geq 3$, they constructed examples with positive velocity and linearly growing fingers, i.e. non-homogenization.  The ideas of their paper will be discussed further below.  Recently Armstrong and Cardaliaguet~\cite{ArmstrongCardaliaguet} have proven that homogenization holds for \eref{level set} in random media with a finite range of dependence property under the Lions-Souganidis condition.  Again the main role of the Lions-Souganidis condition is to guarantee pointwise Lipschitz estimates of the arrival time function.

 The contribution of this paper has two parts.  The first is to point out that the geometric argument of Caffarelli and Monneau~\cite{CaffarelliMonneau} in $d=2$ showing flatness of interfaces in periodic media is in fact a result of regularity theory. In particular it is a large scale (deterministic) Lipschitz estimate for the arrival time function.  The second is to explain that, as is typical in homogenization theory, a large scale regularity result is sufficient to obtain quantitative results.  In particular we are able to adapt the method of Armstrong and Cardaliaguet~\cite{ArmstrongCardaliaguet} to use only large scale Lipschitz regularity. 
 
 In short, our paper extends Armstrong and Cardaliaguet~\cite{ArmstrongCardaliaguet} in random environments, in the same way that Caffarelli and Monneau~\cite{CaffarelliMonneau} extended Lions and Souganidis~\cite{LionsSouganidis} in periodic environments.
 
\begin{theorem}\label{t.main}
Suppose that $c: \R^2 \to (0,+\infty)$ is an $\R^2$-stationary random field with finite range of dependence, almost surely bounded $c_{\min} \leq c(x) \leq c_{\max}$, and Lipschitz continuous $\|Dc\|_{L^\infty(\R^2)} < + \infty$.  Then for every $e \in \R^2$ there is a deterministic asymptotic speed $\bar{c}(e)$ so that the arrival time at point $x$, $m(x)$, of the front started from $\{ x \cdot e \leq 0\}$ satisfies,
\[ \P( | m(te) - \E [m(te)]| > \lambda t^{2/3} ) \leq Ce^{-C^{-1}\lambda^2} \ \hbox{ and } \ |\E[m(te)] - \frac{1}{\bar{c}(e)}t| \leq Ct^{2/3}.\]
Furthermore the effective velocity $\bar{c} : S^1 \to (0,\infty)$ is continuous with logarithmic modulus of continuity.
\end{theorem}

Next we explain the connection of \tref{main} with the homogenization \eref{level set hom ep} to \eref{level set hom} for general initial data.  The general framework for quantitative homogenization of viscous Hamilton-Jacobi equations laid out in \cite{ArmstrongCardaliaguet} appears not to apply due to the extremely weak only logarithmic continuity estimate on $\bar{c}$.  Nonetheless we expect that the standard connection between metric problems and the approximate corrector problem, see Armstrong and Souganidis \cite[Theorem 1]{levelset} or \cite[Proposition 2.4]{ArmstrongCardaliaguet}, still holds in non-quantitative form, and one should be able to obtain almost sure convergence of \eref{level set hom ep} to \eref{level set hom}.  We did not carry out the details here.

\subsection{Literature} Here we give a slightly broader overview of the literature, including the results we have already mentioned.  The first result of homogenization for the forced mean curvature in periodic environments was by Lions and Souganidis \cite{LionsSouganidis}, under their strong coercivity condition guaranteeing Lipschitz estimates.  Dirr, Karali and Yip \cite{DirrKaraliYip} constructed pulsating wave solutions in all dimensions for $V_n = -\kappa + \delta c$ with smooth $c$ (no positivity required) and $\delta$ small.  Cardaliaguet, Lions and Souganidis \cite{CardaliaguetLionsSouganidis} proved homogenization in dimension $2$ with a weak and non-perturbative positivity condition allowing signed $c$, Cesaroni and Novaga \cite{CesaroniNovaga} further weakened this condition in dimension $2$ and also constructed a maximal speed traveling wave in higher dimensional laminar media.  Caffarelli and Monneau \cite{CaffarelliMonneau} constructed counter-examples to homogenization in $d\geq 3$ for some positive $c$, and proved homogenization in $d=2$ for all $c>0$.  As mentioned before, our result is analogous to theirs, but in random media.  In the direction of understanding the nature of non-homogenization, and potentially splitting non-flat traveling front solutions into traveling waves of multiple speeds, Kim and Gao \cite{KimGao} recently showed the existence of head and tail speeds that depend continuously on the normal direction and construct maximal and minimal speed traveling wave solutions in laminar media by a new proof.  

Finally we discuss random media, where the only result for the forced mean curvature flow is by Armstrong and Cardaliaguet \cite{ArmstrongCardaliaguet}.  As described above, they prove homogenization under the strong Lions-Souganidis coercivity condition in all dimensions.  We extend this result in $d=2$ allowing only the weak Caffarelli-Monneau coercivity condition $c >0$.

 \subsection{Physical motivation and related open problems}  The forced mean curvature flow \eref{front} is a model for interface motion in inhomogeneous media, e.g. contact lines of liquid droplets on a rough surface, fluid-fluid phase interface motion in a porous medium, or domain boundaries in random magnetic material \cite{Kardar_1998}.  We consider an (oriented) interface $\Gamma_t$ pushed through the medium by a driving force $F$, moving by normal velocity
 \begin{equation}\label{e.front}
V_n = -\kappa + c(x)+F.
\end{equation}
Pinning defects in the media $c$ compete with surface tension $\kappa$ and large scale forcing $F$ (e.g. contact angle, pressure, applied magnetic field).  When the forcing $F$ is too weak interfaces are pinned, increasing the forcing there is a critical transition and interfaces de-pin and start moving.  

There are many issues which are not clear, at least from a mathematical perspective, about this critical transition.  Define the critical forcing for the transition from zero to positive speed (this may, or may not, be the same as de-pinning)
\[ F_{*,s}(e,c) = \inf \left\{ F: 
\begin{array}{c}
\hbox{the solution of \eref{front} with initial data $\Gamma_0 = \partial \{x \cdot e \leq 0\}$ has} \\
\hbox{$c_*(F,c) = \liminf_{t \to \infty} \inf \frac{1}{t}\Gamma_t \cdot e >0$}
\end{array}
\right\}. \]
 When $F>F_{*,s}$ initially flat fronts separate from their initial data at positive speed.  Now one can ask about the scale of the transversal fluctuations of the interfaces, either moving or pinned.  Physicists have various conjectures on this topic especially at criticality, see for example \cites{Stepanow1994,Erta__1994,Kardar_1998}. The most basic question we can ask in this direction is that of homogenization.  Does \eref{gammahom} hold whenever $F>F_{*,s}(e,c)$?  Or, weaker, when $F> \sup_{e'}F_{*,s}(e',c)$? Or, even weaker, when $F>F_{*,s}(e,c)$ and $c$ is rotation invariant in law?  Do the answers depend on dimension?
 
 A counter-example to homogenization under the first condition can be constructed in $2$-d periodic and laminar medium in the spirit of \cites{CardaliaguetLionsSouganidis,CaffarelliMonneau}.  As explored by \cites{CaffarelliMonneau,CesaroniNovaga,KimGao}, informally speaking, non-homogenization at a given direction should imply pinning at a transversal direction (in $2$-d).  The example of \cite{CaffarelliMonneau} is a counter-example to the second question in $d \geq 3$.  The second question in $2$-d, and the third question in $d\geq 3$ are open.
 

Our result \tref{main} is a step towards addressing this difficult, and more general, open issue.  Our new result is in $2$-d, we do not go all the way down to the pinning transition $\sup_{e'}F_{*,s}(e',c)$, instead we consider the (weakly) coercive case $c(x)+F>0$.  However, since we are able to deal with only large scale Lipschitz estimates (in fact larger than unit scale) for the arrival time, we expect the new ideas developed here to be useful in pushing the analysis down to $F_{*,s}$ where one would at best expect large scale Lipschitz estimates above, now, a random length scale.  Still, there is a huge gap to bridge and we consider this to be a difficult and interesting open question.

\subsection{Acknowledgments}  Thanks Charlie Smart and Pierre Cardaliaguet for helpful conversations.    Thanks to Takis Souganidis for helpful conversations and especially for pointing out the small scale Lipschitz estimate.  Thanks to Inwon Kim for helpful comments on the manuscript. 

\subsection{Support} The author appreciates the support of the Friends of the Institute for Advanced Study and the NSF RTG grant DMS-1246999. 

\section{Set Up and Preliminary Results}

 \subsection{Viscosity Solutions}  We use throughout the paper the notion of viscosity solutions for second order degenerate elliptic (and geometric) equations.  See \cite{GGIS,Sato,CaffarelliMonneau} for proof of comparison principle in this setting.

 \subsection{The random medium}  We lay out the precise assumptions on the random medium.  We require that there are $0<c_{\min} \leq c_{\max} < +\infty$ and $L_0 <+\infty$ so that the following hold
 \begin{equation}\label{e.c assumptions}
  0 < c_{\min} \leq c(x) \leq c_{\max} \ \hbox{ for all $x \in \R^d$ and } \ \|Dc\|_\infty \leq L_0.
  \end{equation}
For concreteness, our probability space $\Omega$ can be taken as the collection of all such coefficient fields
 \[ \Omega := \{ c: \R^d \to (0,\infty) : \ \hbox{\eref{c assumptions} holds} \}.\]
 We associate with $\Omega$ a family of (cylinder) $\sigma$-algebras $\mathcal{F}(U)$ for $U \subset \R^d$ a Borel set,
 \[ \mathcal{F}(U): = \sigma(c \mapsto c(x): x \in U)\]
 The largest of these $\sigma$-algebras is $\mathcal{F}(\R^d)$ which, if we refer to it, will be just called $\mathcal{F}$.  The underlying physical space $\R^d$ naturally acts on $\Omega$ by translations, for each $x \in \R^d$ we define $T_y : \Omega \to \Omega$ by,
 \[ (T_yc)(\cdot) := c(\cdot+y).\]
 One can check easily that this is indeed a group action.
 
 \medskip
 
 Now we suppose that we are given a probability measure $\P$ on the measurable space $(\Omega,\mathcal{F})$, which satisfies the following properties:
 \begin{enumerate}[label = $\circ$]
 \item Stationarity: for every $y \in \R^d$ and every $E \in \Omega$,
 \[ \P(E) = \P(T_y(E)).\]
 \item $1$-dependence: for every two Borel sets $U,V$ of $\R^d$ with $\textup{dist}(U,V) \geq 1$,
 \[ \mathcal{F}(U) \ \hbox{ and }  \ \mathcal{F}(V) \ \hbox{ are $\P$-independent.}\]
 \end{enumerate}
 The regularity estimates that we prove in Section~\ref{sec: Properties of the Arrival Time} are uniform in $\Omega$ and therefore are not probabilistic in nature.  The assumptions of stationarity and $1$-dependence of the probability measure $\P$ will come into play in the remainder of the paper Sections~\ref{sec: Estimate of the Random Fluctuations}-\ref{s.deterministic}.  
 
 \subsection{Notations} Constants which depend at most on $c_{\min}$, $c_{\max}$, $d$ and $L_0$ will be called universal and will be written usually as $C$.  The value of $C$ may change from line to line.

  \section{Properties of the Arrival Time}\label{sec: Properties of the Arrival Time}
In this section we introduce the arrival time problem, explain its relation with the front propagation problem \eref{level set}, and prove some fundamental regularity properties.  The main new result in this section is a large scale (larger than unit scale) Lipschitz estimate of the arrival time function.  As explained in the introduction this result is the key to the remainder of the paper.
 
 \subsection{The arrival time problem}  In this section we introduce the arrival time problem starting from a nonempty closed set $S \subset \R^d$ with smooth boundary satisfying the following $C^{1,1}$ regularity/smallness condition.
 \begin{assumption}\label{as.Sassumption}
$S$ has interior tangent balls of radius $R_0  \geq 1$ and exterior tangent balls of radius $1$ at every boundary point.
 \end{assumption}
For now $R_0$ is a free parameter in the regularity condition, but we will soon fix $R_0 = \frac{C(d)}{c_{\min}} \vee 2$, chosen so that a ball of radius $R_0/8$ moving with velocity at most $c_{\min}/2$ in any direction is a subsolution of \eref{level set}.  We will keep that value for the remainder of the paper.

 The arrival time is the maximal subsolution of,
 \begin{equation}\label{e.arrival time}
 \left\{
 \begin{array}{lll}
 -\tr\left[(I - \tfrac{Dm \otimes Dm}{|Dm|^2})D^2m\right] + c(x) |Dm| =  1 & \hbox{ in } & \R^d \setminus S \vspace{1.5mm}\\
 m = 0 & \hbox{ on } & S.
 \end{array}\right.
 \end{equation}
 When we wish to emphasize all the dependencies of $m$ we write $m(x,S,c)$. 
 
  Note that this problem is typically referred to as the {metric problem} in the Hamilton-Jacobi literature \cite{ArmstrongCardaliaguet, ArmstrongSouganidis}.  In the case of the mean curvature flow it is more natural to think of $m(x,S)$ as the first time that the evolving region started from $S$ hits the point $x$.  For this reason we call $m$ the arrival time.  To clarify this connection between the arrival time problem and the front propagation problem we make note of the following:
 \begin{lemma}
 Let $m$ solve the arrival time problem \eref{arrival time}.  Then
 \[ u(x,t) =t- t \wedge m(x)  \]
 solves the level set equation for the forced mean curvature flow \eref{level set}.
 \end{lemma}
 The zero level set of $u(x,t) = t-t \wedge m(x)  $ can be interpreted as the locus of the front started from $\partial S$ and evolving under \eref{front}.  
 
The following theorem summarizes the results we will prove in this section.

\begin{theorem}\label{t.mainreg}
Suppose that $S$ satisfies \asref{Sassumption}.  There exists a unique solution $m(x) = m(x,S)$ of \eref{arrival time} satisfying:
\begin{enumerate}[label=(\roman*)]
\item\label{part.sslip} {(Small scale Lipschitz)} For $x,y \in \R^d \setminus S$
 \[ |m(x) - m(y)| \leq \tfrac{2}{c_{\min}}e^{\| \grad c\|_\infty m(x) \wedge m(y)}|x-y|.\]
    \item\label{part.settimereg} {(Regularity in time)} Call $S_t = \{m(x) \leq t\}$ for any $s,t \geq 0$
   \[ d_H(S_t,S_s) \leq C(d)|t-s|^{1/2}+2c_{\max}|t-s|.\]
    \item\label{part.setreg} {(Regularity w.r.t$.$ the data)} For $S$ and $S'$ both satisfying \asref{Sassumption}
  \[ |m(x,S) - m(y,S')| \leq \tfrac{2}{c_{\min}}d_H(S,S'). \]
 \end{enumerate}
 Additionally the following property holds in $d=2$:
 \begin{enumerate}[label=(\roman*),resume]
 \item {(Large scale Lipschitz)}\label{part.largescalelip} There exist $\tau,L >0$ depending on $c_{\min}$ such that for all $x,y \in \R^2 \setminus S$
 \[ |m(x) - m(y)| \leq \tau + L|x-y|.\]
\end{enumerate}
\end{theorem}

\subsection{Local regularity}
Let $S \subset \R^d$ be a compact set with smooth boundary and having interior tangent balls of radius $R_0  \geq 1$ and exterior tangent balls of radius $1$ at every boundary point.  Consider the evolution
 \begin{equation}\label{e.evolution}
 u_t = \tr[(I - \tfrac{Du \otimes Du}{|Du|^2})D^2u]+c(x)|Du|  \ \hbox{ in } \ \R^d \times (0,\infty) 
 \end{equation}
 with
 \begin{equation}\label{e.evolutiondata}
 u(x,0) = -d(x,S) \ \hbox{ in } \ \R^d
 \end{equation}
 We will consider the existence, uniqueness and local regularity of the solutions of \eref{evolution}.  We will also be interested in the regularity of the corresponding arrival time.
 
 \begin{proposition}[Existence and local regularity]\label{p.local reg}
Let $S \subset \R^d$ satisfying \asref{Sassumption} with $R_0 \geq \frac{2(d-1)}{c_{\min}} \vee 2$.  There exists a unique solution $u$ of \eref{evolution} which is continuous with modulus
 \[|u(x,t) - u(y,t)| \leq \tfrac{2}{c_{\min}}e^{\|Dc\|_{\infty}t}|x-y| \ \hbox{ for all } \ x,y \in \R^d \setminus S, \ t\in [0,\infty),\]
 and monotone increasing (strictly monotone in its negativity set)
 \[ u_t(x,t) \geq \frac{1}{2}c_{\min}{\bf 1}_{\{u <0\}}   \ \hbox{ in } \ \R^d \times [0,\infty).\]
 \end{proposition}
 
 This result is similar to Proposition 5.1 in \cite{CaffarelliMonneau}, we need to deal with a more general class of initial data.  The other new element here is the quantification of the local regularity of $u$ which is guaranteed by the comparison principle.   This idea goes back to \cite{Souganidis, CrandallLions} for Hamilton-Jacobi equations and is surely known for mean curvature type equations although we do not have a reference.  We give a slightly different proof based on a more geometric approach using inf-convolutions.  Essentially the idea is to compare $u(x,t)$ with a translation $v(x,t) = u(x+\xi,t)$, but these do not solve the same equation since,
 \begin{align*}
   v_t &= \tr[(I - \tfrac{Dv \otimes Dv}{|Dv|^2})D^2v]+c(x+\xi)|Dv| \\
   &\geq \tr[(I - \tfrac{Dv \otimes Dv}{|Dv|^2})D^2v]+(c(x)-\|Dc\|_\infty|\xi|)|Dv|.
   \end{align*}
 In order to fix this we speed up the level sets of $v$ by an inf-convolution over balls $B_{r(t)}$ of decreasing radius.
 
 \begin{lemma}\label{l.infconvolution}
 Suppose that $u$ is a solution of \eref{evolution} on $\R^d \times (0,T)$ and $r: (0,T) \to (0,\infty)$ is $C^1$, then
 \[ v(x,t) = \inf_{z\in B_{r(t)}(x)}u(z,t)\]
 is a supersolution on $\R^d \times (0,T)$ of 
 \[ v_t \geq \tr[(I - \tfrac{Dv \otimes Dv}{|Dv|^2})D^2v] + (\inf_{z\in B_{r(t)}}c(x+z)-r'(t))|Dv|.\]
 \end{lemma}
We omit the proof of this lemma since it is standard and proceed with the proof of \pref{local reg}.

 \begin{proof}[Proof of \pref{local reg}]
 1. We start with an auxiliary problem
  \begin{equation}\label{e.evolutionB}
 \left\{
 \begin{array}{lll}
 \tilde{u}_t = \tr[(I - \tfrac{D\tilde{u} \otimes D\tilde{u}}{|D\tilde{u}|^2})D^2\tilde{u}]+c(x)|D\tilde{u}|  & \hbox{ in } & \R^d \setminus S \times (0,\infty) \vspace{1.5mm}\\
 \tilde{u}(x,0) = u_0(x) :=-d(x,S) & \hbox{ in } & \R^d \setminus S \\
 \tilde{u}(x,t) = t\left[\tr[(I - \tfrac{Du_0 \otimes Du_0}{|Du_0|^2})D^2u_0]+c(x)|Du_0|\right] & \hbox{ on } & \partial S \times (0,\infty)
 \end{array}\right.
 \end{equation}
  Define the barrier sub and supersolutions,
 \[ u^\pm(x,t) = u_0(x) +C_\pm t \ \hbox{ with $C_- = \frac{1}{2}c_{\min}$ and $C_+ = (d-1)+c_{\max}$.}\]
 Let $x \in \partial \{ u_0 \geq \lambda\}$ for any $\lambda \leq 0$, let $y$ be the nearest point to $x$ in $\partial S$.  Then $S$ has an interior ball $B_{R_0}(y')\subset S$ centered at some $|y'-y| = R_0$.  Then $B_{R_0}(y'+(x-y))$ is an interior ball to $\{ u_0 \geq \lambda\}$ at $x$.  Thus (in the viscosity sense)
 \[ \tr[(I - \tfrac{Du_0 \otimes Du_0}{|Du_0|^2})D^2u_0](x) \geq -\frac{d-1}{R_0}\]
 and so the following holds (in the viscosity sense)
 \[  -\frac{d-1}{R_0}+c_{\min} \leq \tr[(I - \tfrac{Du_0 \otimes Du_0}{|Du_0|^2})D^2u_0]+c(x)|Du_0| \ \hbox{ for } \ x \in \R^d \setminus S.\]
 The left hand side is larger than $C_-=\frac{1}{2}c_{\min}$ for $R_0\geq 2(d-1)/c_{\min}$.  Also, by the unit exterior ball condition of \asref{Sassumption},
 \[ \tr[(I - \tfrac{Du_0 \otimes Du_0}{|Du_0|^2})D^2u_0]+c(x)|Du_0| \leq (d-1) +c_{\max} = C_+ \ \hbox{ on } \ \partial S. \]
 Thus $u^{\pm}$ are respectively sub and supersolutions of \eref{evolutionB} with ordering holding on $\partial S$ as well by the same argument.  Then by Perron's method there exists a solution $\tilde{u}$ of \eref{evolutionB} with,
 \[ u^- \leq \tilde{u} \leq u^+.\]
See for more details on the application of Perron's method for this equation \cite{ChenGigaGoto,CrandallIshiiLions}.
 
 \medskip
 
 2. Now we prove the regularity in time.  Let $h\geq 0$ we have
 \[ \tilde{u}(x,t+h) \geq \tilde{u}(x,t) +C_-h,\]
  so by comparison principle the same holds for all $t \geq 0$.  This gives 
 \[\frac{1}{2}c_{\min}\leq \tilde{u}_t.\]
Now since $\tilde{u} > 0$ on $\partial S$ for all $t >0$ we have actually that  $u = \tilde{u} \wedge 0$ solves \eref{evolution} on $\R^d$, and by comparison it is the unique solution.  Therefore
\[ u_t(x,t) \geq \frac{1}{2}c_{\min}{\bf 1}_{\{u(x,t) <0\}}\]
since $u$ and $\tilde{u}$ agree on $\{ u(x,t)<0\}$.
 \medskip
 
 3.  Finally we address the local regularity in $x$.  That $u$ is continuous is simply a consequence of uniqueness.  Let $\xi \in \R^d$ and $t_0>0$, we aim to estimate,
 \[ u(x,t_0) - u(x+\xi,t_0) \leq e^{\|Dc\|_\infty t_0}|\xi|.\]
To this purpose we consider the geometric inf-convolution
 \[ v(x,t) = e^{\|Dc\|_\infty t_0}|\xi| +  \inf_{z \in B_{r(t)}} u(x+z,t) \ \hbox{ with } \ r(t) = e^{-\|Dc\|_\infty (t-t_0)}|\xi|.\]
   We will check that $v$ is a supersolution of \eref{evolution}. First check the initial data, using the $1$-Lipschitz condition of the distance function $u_0$,
 \[  v(x,0) = e^{\|Dc\|_\infty t_0}|\xi| + \inf_{z \in B_{e^{\|Dc\|_\infty t_0}|\xi|}} u_0(x+z)  \geq u_0(x).\]
 By \lref{infconvolution}, with $r(t) =e^{-\|Dc\|_\infty (t-t_0)}|\xi|$, $v$ is a supersolution of
 \begin{align*} 
 v_t &\geq \tr[(I - \tfrac{Dv \otimes Dv}{|Dv|^2})D^2v] + (r'(t)+\inf_{z\in B_{r(t)}}c(x+z))|Dv| \\
 & \geq \tr[(I - \tfrac{Dv \otimes Dv}{|Dv|^2})D^2v] + (c(x) - \|D c\|_\infty r(t) -r'(t))|Dv| \\
 & = \tr[(I - \tfrac{Dv \otimes Dv}{|Dv|^2})D^2v] + c(x)|Dv|
 \end{align*}
  since $r'(t) = - \|Dc\|_\infty r(t)$.  
  
  By comparison principle the ordering between $u$ and $v$ persists and we have for all $x \in \R^d$ at time $t_0$,
 \[ u(x,t_0) \leq v(x,t_0) = e^{\|Dc\|_\infty t_0}|\xi| +\inf_{z \in B_{|\xi|}} u(x+z,t_0) \leq e^{\|Dc\|_\infty t_0}|\xi| +u(x+\xi,t_0).\]
 \end{proof}
Now we return to derive the first three parts of \tref{mainreg}.  The local Lipschitz regularity of the arrival time, part \partref{sslip} of \tref{mainreg}, is a corollary of \pref{local reg}.  

 \begin{proof}[Proof of \tref{mainreg} parts \partref{sslip}-\partref{settimereg}]  Part \partref{sslip}.  We just need to show that if $u(x,t)\geq 0$ then, for $y \in \R^d$ with $m(y) \geq m(x)$,
 \[ u(y,t+h) \geq 0 \ \hbox{ for } \ h \geq  \tfrac{2}{c_{\min}}e^{\|Dc\|_{\infty}t}|x-y|.\]
 For this we combine the Lipschitz regularity of $u$ in space with the lower bound $u_t \geq \frac{1}{2}c_{\min}{\bf 1}_{\{u<0\}}$ from \pref{local reg}.  We have, either $u(y,t+h) \geq 0$ and we are done, or
 \[ u(y,t+h) \geq u(y,t) + \tfrac{1}{2}c_{\min}h \geq u(x,t) + \tfrac{1}{2}c_{\min}h - e^{\|Dc\|_{\infty}t}|x-y|,\]
 the sum of the last two terms being nonnegative for $h \geq \frac{2}{c_{\min}}e^{\|Dc\|_{\infty}t}|x-y|$ as claimed.
 
 Part \partref{setreg}.  Noting that $m(x,S \cap S') \geq m(x,S), m(x,S') \geq m(x,S \cup S')$, we can assume without loss that $S \subset S'$.  Then applying part \partref{sslip} for $x \in S'$
 \[ m(x,S) \leq \inf_{y \in S}\frac{2}{c_{\min}}e^{\|Dc\|_{\infty}m(x,S) \wedge m(y,S)}|x-y| \leq \frac{2}{c_{\min}}d_H(S,S')\]
 using that $m(y,S) = 0$ for $y \in S$.  Then $m(x,S) - \frac{2}{c_{\min}}d_H(S,S')$ is a subsolution of \eref{arrival time} for $S'$ and so $m(x,S') \geq m(x,S)$ since $m(x,S')$ is the maximal subsolution.
 
 Part \partref{settimereg}.  Assume without loss that $t \geq s \geq 0$, we can also take $s = 0$ at the cost that $C^{1,1}$ regularity \asref{Sassumption} is lost.  Let $\rho_\ep$ be a standard mollifier $M$ to be chosen and define
 \[ w(x) = \tfrac{1}{M}[\rho_\ep \star d(\cdot,S)](x) - \tfrac{C(d)\ep}{M}.\]
 Then $w(x) \leq 0$ on $S$, for $C(d)$ above sufficiently large, and
 \[ |\grad w(x)| \leq \frac{1}{M}, \ |D^2 w(x)| \leq \frac{C(d)}{M\ep}, \ \hbox{ and } \ |w(x) - \tfrac{1}{M}d(x,S)| \leq \frac{C(d)\ep}{M}.\]
 Choose $\ep = t^{1/2}$.  Now $w$ is a subsolution of 
 \[ - \tr[(I - \tfrac{Dw \otimes Dw}{|Dw|^2})D^2w] +|\grad w|\leq \frac{C(d)}{M\ep} + \frac{c_{\max}}{M} \leq 1\]
when we choose $M = 2c_{\max}+2C(d)t^{-1/2}$. Thus by comparison in $\R^d \setminus S$ we have $m(x) \geq w(x)$ and therefore
 \[ \{m(x) \leq t\} \subset \{ w(x) \leq t\} \subset   \{ d(x,S) \leq  C(d)t^{1/2}+2c_{\max} t\}.\]
\end{proof}
 \subsection{Large Scale Lipschitz Estimates for the Arrival Time in $2$-d}
 Finally we return to the main new part of \tref{mainreg}, the large (unit) scale Lipschitz estimate of the arrival time function. This is a reinterpretation of the flatness result of Caffarelli and Monneau \cite{CaffarelliMonneau} in periodic media.  Let $S \subset \R^2$ satisfying the regularity conditions \asref{Sassumption} and let $u(x,t)$ be the corresponding solution of \eref{evolution}.

 \begin{theorem}\label{t.large scale lip}
 Suppose that $S$ satisfies \asref{Sassumption}.  There exists $\tau, L>0$ depending only on $c_{\min}$ so that for every every $x,y \in \R^2 \setminus S$,
 \[ |m(x,S) - m(y,S)| \leq \tau + L|x-y|.\]
 Precisely $\tau = C(d)c_{\min}^{-1}\min\{c_{\min},1\}^{-1}$ and $L = \tau/R_0= C(d)c_{\min}^{-1}$.
 \end{theorem}
 What is actually proven in \cite{CaffarelliMonneau} Proposition 6.1 is (with only a very small modification) the following essential result.
 \begin{proposition}\label{p.filling time}
 Suppose that $S$ satisfies \asref{Sassumption}.  There is a waiting time $\tau = \frac{13R_0}{2c_{\min}}$ so that for any $x_0 \in \R^2$ such that $m(x_0,S) \leq t$
 \[m(x,S) \leq t + \tau \ \hbox{ for all } \ x \in B_{R_0}(x_0).\]
 \end{proposition}
 From \pref{filling time} it is simple to prove \tref{large scale lip}.  Let $x,y \in \R^d \setminus S$ and call $e = \frac{x-y}{|x-y|}$,
 \[ m(y,S) \leq m(y + R_0 e,S) + \tau \leq \cdots \leq m(x,S) + \lceil \tfrac{|x-y|}{R_0}\rceil \tau \leq m(x,S) + \tau+\tfrac{\tau}{R_0}|x-y| ,\]
 which is exactly the estimate claimed by \tref{large scale lip}.
 
 First we prove a lemma on the path-connectedness of the set $S_t$.  This is almost the same as Proposition 5.7 of \cite{CaffarelliMonneau}, but we need to consider more general initial data which are not just half-spaces. 
 \begin{lemma}\label{l.connectedR}
 Let $x \in S_t$ and let $\omega_0$ be any path-connected component of $\textup{int}(S_t)$ containing $x$ in its closure.  Let $0 <r < R_0$ then $\omega_0 \cap \partial B_r(x_0) \neq \emptyset$.
 \end{lemma}
 \begin{proof}
Suppose that $\omega_0 \subset B_r(x)$ for some $0 < r < R_0$.  Then $\omega_0 \cap S_0 = \emptyset$ since every connected component of $\textup{int}(S_0)$ contains a ball of radius $R_0$ by \asref{Sassumption}.  

Call $t_* = \inf \{s \in [0,t]: S_s \cap \omega_0 \neq \emptyset\}$.  First note that 
\[ \emptyset \neq S_{t_*} \cap \bar{\omega}_0 \subset \partial \omega_0 \cap \partial S_{t_*},\]
 this is just a consequence of the definition of $t_*$ and the continuity in time of $S_s$ in Hausdorff distance by \tref{mainreg} part \partref{settimereg}.  In particular $t_* < t$.  Let $x_0 \in \partial S_{t_*} \cap \partial \omega_0$, then by \tref{mainreg} part \partref{sslip} we have $x_0 \in \textup{int} (S_s)$ for all $s > t_*$.  This fact applied to $S_{t}$ contradicts that $\omega_0$ is a connected component of $\textup{int} (S_t)$.
 
 \end{proof}
 
\begin{proof}[Proof of \pref{filling time}.]
 This is basically Proposition 6.1 in \cite{CaffarelliMonneau}, we just mention the differences.  Let $x_0$ with $m(x,S) = t_0$.  Work in $B_{R}(x_0)$ for any $R_0/2 \leq R < R_0$ so that \lref{connectedR} applies and provides a path $\gamma$ contained in $S_t \cap B_{R}(x_0)$ connecting $x_0$ to $\partial B_{R}(x_0)$.  By the choice of $R_0$ the self-propagating ball barrier
 \[ \varphi_z(x,t) = {\bf 1}_{B_{R/4}(z(t))}(x) \ \hbox{ with for any path $z: \R \to \R^2 $ with } \ |\dot{z}(t)| \leq c_{\min}/2\]
 is a subsolution of \eref{evolution}.  Then follow the proof of Proposition 6.1 in \cite{CaffarelliMonneau} through Steps 1 and 2.  The result is that $B_{R/4}(z_0) \subset S_{t_0 + \tau'}$ where $z_0  \in \partial B_{R/2}(x_0)$ and $\tau' = 5R_0/c_{\min}$.  Then use the subsolutions $\varphi_z$ with 
 \[ z(t) = z_0 + \tfrac{1}{2}c_{\min}(t-t_0-\tau')\xi \ \hbox{ for each }  \xi \in S^{d-1}\]
 on the time interval $t \in [t_0+\tau',t_0+\tau'+\frac{3R_0}{2c_{\min}}]$ to conclude that $B_{R_0}(x_0) \subset S_{t+\frac{13R_0}{2c_{\min}}}$.
  
\medskip

\end{proof}
 
 We also mention here a corollary of the large scale Lipschitz estimate, which is the large scale strict monotonicity in time of the evolution $S_t$.

  \begin{lemma}\label{l.sub-level cont}
 For every $0 \leq s \leq t$,
  \[ S_s + B_{\frac{1}{L}((t-s)-\tau)_+}\subseteq S_t  .\]
 \end{lemma}
 
  \begin{proof}
Let $x \in S_s$, by the large scale Lipschitz estimates \tref{large scale lip}
\[ m(y,S) \leq m(x,S) + \tau + L|y-x| \leq s + \tau + L|y-x| \leq t\]
as long as $|y-x| \leq \frac{1}{L}((t-s)-\tau)_+$.  In other words,
\[ \{m(\cdot ,S) \leq s\} + B_{\frac{1}{L}((t-s)-\tau)_+}\subseteq\{m(\cdot ,S) \leq t\},\]
which was the desired result.

 \end{proof}

 \section{Estimate of the Random Fluctuations}\label{sec: Estimate of the Random Fluctuations}
 In this section we consider the random part of the error estimate, for a given unit direction $e$ and corresponding half-space $H(e) = \{x \cdot e \leq 0\}$ we show that,
 \[ |m(x,H(e)) - \E [m(x,H(e))]| \lesssim (x \cdot e)^{1/2} \ \hbox{ with high probability.}\]
 In other words, the random fluctuations of the arrival time are quantitatively of lower order than the size of the arrival time itself which is of order $(x \cdot e)$.
 
\begin{proposition}\label{p.fluctuations}
Suppose that $d=2$ and $S$ satisfies \asref{Sassumption}, then for all $d(x,S) \geq 1$
\[ \log \P(|m(x,S) - \E m(x,S)| > \lambda d(x,S)^{1/2}) \leq  - \frac{c_{\min} \min\{1,c_{\min}^2\}}{Cc_{\max}} \lambda^2+C,\]
where $C$ are numerical constants.
\end{proposition}
 
 This idea of the proof is very similar to \cite{ArmstrongCardaliaguetSouganidis,ArmstrongCardaliaguet} (specifically see \cite[Proposition 3.1]{ArmstrongCardaliaguet}). The proof only relies on large scale Lipschitz estimates not local regularity.  For completeness we include the adapted proofs in \aref{randomfluct}.  There are some small adjustments needed in the arguments since our Lipschitz estimate does not follow directly from the arrival time PDE, but rather from the connection with a front propagation initial data problem problem in the whole space (in particular no boundary).

 \section{Estimate of the Deterministic Part of the Error}\label{s.deterministic}
In this section we consider the deterministic part of the error estimate, aiming to show, for $H(e) = \{x \cdot e \leq 0\}$,
\begin{equation}
 \left|\E [m(x,H(e))] - \frac{1}{\bar{c}(e)}(x \cdot e)\right| \lesssim ( 1+(x \cdot e)^{2/3} ).
 \end{equation}
 As in \cite{ArmstrongCardaliaguet} the idea is to show that $\E[m(te,H(e))]$ is approximately linear in $t$ by using the concentration estimate along with a quantitative (localized) uniqueness property of the arrival time.  
 
 Here the lack of local regularity is a more serious issue.  On its face, the localized uniqueness property, originally proved in \cite{ArmstrongCardaliaguet}, really relies on local regularity of the solution. The only local regularity that we have is the Lipschitz estimate of \tref{mainreg} part \partref{sslip} which grows exponentially in time.  If we were to directly emulate the localization result of \cite{ArmstrongCardaliaguet} then, in order to obtain ordering between a pair of solutions at time $t$,  would require ordering of the solutions at time $0$ on a region of size $e^{Ct}$.  This exponential growth is, just barely, too large of a region to control using the fluctuations estimate \pref{fluctuations}.  We provide a slight improvement on this exponential localization scale by taking advantage of the large scale Lipschitz estimate. Basically we regularize $n$ times during the evolution, resulting in an additional error of size $\sim n$, but reducing the time interval on which we need to use local regularity so that the localization rate is only $ne^{Ct/n}$.  This result can be applied with $n \sim t^{\alpha}$ for some $\alpha \in (0,1)$ chosen judiciously.
 
 \medskip

\subsection{Localized influence of the boundary data}  
Our first goal is to establish the localization property of the arrival time solutions.  
\begin{proposition}\label{p.localization}
Let $m^1$ and $m^2$ be solutions of the arrival time problem,
\[ - \tr\left[(I - \tfrac{Dm \otimes Dm}{|Dm|^2})D^2m\right] + c(x) |Dm| = 1 \ \hbox{ in } \ \R^2 \setminus S^i\]
with $m^i = 0$ in $ S^i$.  Suppose that: $m^i$ both satisfy the unit scale Lipschitz estimate \tref{large scale lip} with constants $\tau,L \geq 1$, there is $R \geq 1$ such that the ordering holds
\[  S^2 \subset S^1 \ \hbox{ on } \   B_R(0).\]
  Then there exists $C_1(\|Dc\|_\infty,c_{\min}),C_2(c_{\min}) \geq 1$ such that, if $s \geq 1$, $n \leq s$, and $R \geq \bar{R}_n(s) = ne^{C_1(\tau+L\frac{s}{n})}$,
\[ \{m^2(x) \leq s-C_2n\} \subset \{m^1(x) \leq s\}   \ \hbox{ on } \ B_{R-\bar{R}_n(s)}(0).\]
\end{proposition}
Note that we only assume that $m^i$ satisfy the conclusion of \tref{large scale lip}, not that the $S^i$ satisfy \asref{Sassumption}.  Of course we will apply this result when $S^i$ arise as some sub-levels $S^i = \{m(x,S^i_0) \leq t\}$ of the arrival time starting from a regular set $S^i_0$ so that \tref{large scale lip} does apply.

This is a generalization of the localized uniqueness of planar fronts proved in \cite[Proposition 4.2]{ArmstrongCardaliaguet} under a Lipschitz bound.  In fact we will derive the Proposition by iterating \cite[Proposition 4.2]{ArmstrongCardaliaguet}, a variant of which we recall here.
\begin{lemma}\label{l.localization}
Let $m^1$ and $m^2$ be solutions of the arrival time problem,
\[ - \tr\left[(I - \tfrac{Dm \otimes Dm}{|Dm|^2})D^2m\right] + c(x) |Dm| = 1 \ \hbox{ in } \ \R^2 \setminus S^i\]
with $m^i = 0$ in $ S^i$.  Suppose that both $S^i$ satisfy \asref{Sassumption}, there is $R \geq 1$ such that the ordering holds
\[  S^2 \subset S^1 \ \hbox{ on } \   B_R(0).\]
  Then there exists $C_1(\|Dc\|_\infty,c_{\min})\geq 1$ such that, if $s \geq 1$ and $R \geq \bar{R}(s) = e^{C_1s}$,
\[ \{m(x,S^2) \leq s-1\} \subset \{m(x,S^1) \leq s\}   \ \hbox{ on } \ B_{R-\bar{R}(s)}(0).\]
\end{lemma}
The proof is similar to \cite[Proposition 4.2]{ArmstrongCardaliaguet} so we postpone it to \aref{localization}.

 For $m$ and $S$ as in the \pref{localization}, each sublevel set $S_t=\{m(x) \leq t\}$ can be replaced by a regular set $\tilde{S}$ satisfying \asref{Sassumption} up to a constant size error in the arrival time.  This is a key consequence of the large scale Lipschitz estimate.
 
 \begin{lemma}\label{l.m bar close}
 Suppose that $m(x,S)$ satisfies the conclusion of \tref{large scale lip}.  Then there exists $\tilde{S}$ with $S \subset \tilde{S}$ and $d_H(\tilde{S},S) \leq R_0+2$ and
 \[ m(x,\tilde{S}) \leq m(x,S) \leq m(x,\tilde{S})+ 3\tau. \]
 \end{lemma}
 \begin{proof}
Take
\begin{equation}\label{e.m bar}
{S}' = \bigcup_{z\in B_{R_0+1}} (S+z) \ \hbox{ and } \ \tilde{S} = \bigcap_{z\in B_{1}} (S'+z).
 \end{equation}
Then $S \subset \tilde{S}$ and $d_H(\tilde{S},S) \leq 2R_0$ and $\tilde{S}$ satisfies \asref{Sassumption}.   Then by the assumed Lipschitz estimate
\begin{equation}\label{e.m bar close}
m(x,\tilde{S}) \leq m(x,S) \leq  m(x,\tilde{S})+ \tau +2 L R_0. 
 \end{equation}
Note that the second inequality follows from the Lipschitz estimate on $\partial \tilde{S}$ and then by comparison in $\R^2 \setminus \tilde{S}$. Also note that we must apply \tref{large scale lip} here and not \tref{mainreg} part \partref{setreg} since $S$ does not satisfy \asref{Sassumption}.  Finally recall that $L = \tau/R_0$ to conclude with the claimed constant.
 \end{proof}

\begin{proof}[Proof of \pref{localization}]
We prove by inductively applying \lref{localization}.  Suppose that the bound holds for any $0 \leq s' \leq \tfrac{ks}{n}$
\begin{equation}\label{e.inducthyp}
 \{m(x,S^2) \leq s'-(1+3\tau)k\} \subset \{m(x,S^1) \leq s'\} \ \hbox{ in } \ B_{R - k\bar{R}(\frac{s}{n})}(0)
 \end{equation}
for some $0 \leq k \leq n-1$ where $\bar{R}$ is from \lref{localization} and $\tau$ is from \tref{large scale lip}.  Note that the ordering does hold in the case $k=0$ by the assumption of the Proposition.

Now we wish to apply \lref{localization}.  Call  
\[ \Sigma^1 = \{m(x,S^1) \leq \frac{ks}{n}\} \ \hbox{ and } \ \Sigma^2 = \{m(x,S^2) \leq \frac{ks}{n}-(1+3\tau)k\}.\]
By the inductive assumption
  \begin{equation}\label{e.containk}
    \Sigma^2 \subset \Sigma^1 \ \hbox{ in } \ B_{R - k\bar{R}(\frac{s}{n})}(0).
    \end{equation}
  Now regularize replacing $\Sigma^j$ by $\tilde{\Sigma}^j$ as in \eref{m bar} at the cost of the error, from \lref{m bar close},
\begin{equation}\label{e.sigmabarclose}
m(x,\tilde{\Sigma}^j) \leq m(x,\Sigma^j) \leq m(x,\tilde{\Sigma}^j) + 3\tau.
\end{equation}
Also \eref{containk} still holds for the $\tilde{\Sigma}^j$ since the operation $\Sigma \mapsto \tilde{\Sigma}$ preserves containment.  The $\tilde{\Sigma}^j$ now satisfy \asref{Sassumption}.  Now for each $x_0 \in  \cap B_{R - (k+1)\bar{R}(\frac{s}{n})}(0)$ apply \lref{localization} in $B_{R -k\bar{R}(\frac{s}{n})- |x_0|}(x_0)$ to find
\[  \{m(x,\tilde{\Sigma}^2) \leq s'-1\} \subset \{m(x,\tilde{\Sigma}^1) \leq s'\} \ \hbox{ in } \ B_{R - (k+1)\bar{R}(\frac{s}{n})}(0) \ \hbox{ for } \ 0 \leq s' \leq \frac{s}{n}. \]
Now recalling the bounds from \eref{sigmabarclose}, the following containments hold in $B_{R - (k+1)\bar{R}(\frac{s}{n})}(0)$
\begin{align*}
 \{m(x,S^2) \leq \frac{ks}{n} + s' -(1 + 3\tau)(k+1)\} &= \{m(x,\Sigma^2) \leq s'-1 - 3\tau\} \\
 &\subset \{m(x,\tilde{\Sigma}^2) \leq s'-1-3\tau\} \\
 & \subset \{m(x,\tilde{\Sigma}^1) \leq s' - 3\tau\} \\
 &\subset \{m(x,{\Sigma}^1) \leq s' \} \\
 & = \{m(x,S^1) \leq \frac{ks}{n}+s'\}.
\end{align*}
for any $3\tau \leq s' \leq \frac{s}{n}$.  The case $0 \leq s' \leq 3\tau$ already follows from \eref{inducthyp}.  Thus \eref{inducthyp} holds for $k+1$ and we conclude by induction.
\end{proof}

\subsection{Approximate linearity and the convergence of expectations}
In this section we are finally able to address the convergence of the expectations $\E[m(te,H(e))]$ as $t \to \infty$.  Since here we will just consider half-space initial data with a fixed $e$ we will simply refer to $m(x) = m(x,H(e))$.   

\medskip

First we use the concentration estimate of \pref{fluctuations} in combination with the large scale Lipschitz estimate and a union bound to estimate the probability that $|m(x) - \E m(x)|$ is too large {anywhere} on a very large region of $x \cdot e = t$.  To quantify this we define,
\[ M_R(t) := \sup_{x \in B_{Rt} \cap \{x \cdot e \leq t\}} |m(x) - \E[m(x)]|.\]
Note that by stationarity $\E[m(x)]$ is constant on $x\cdot e = t$.
\begin{lemma}\label{l.union bd}
There exists $C(c_{\min},c_{\max}) \geq 1$ so that for every $R,t \geq 2$,
\[\E[M_R(t)] \leq C t^{1/2}\log^{1/2}(Rt).\]
\end{lemma}

Next we use \lref{union bd} and \pref{localization}.
\begin{lemma}\label{l.almost linear}
There exists $C(c_{\min},c_{\max},\|\grad c\|_\infty) \geq 1$ such that for every $t,s\geq 1$,
\[|\E m(te) + \E m(se) - \E m((t+s)e)| \leq Ct^{2/3} \]
\end{lemma}
With the approximate linearity it is relatively standard to show the convergence of $\frac{1}{t}\E[m(te)]$.  
\begin{lemma}\label{l.deterministicconv}
For each $e \in S^{d-1}$ there exists $\bar{c}(e)$ such that
\[ \left| \frac{1}{t}\E[m(te)] - \frac{1}{\bar{c}(e)} \right| \leq Ct^{-1/3}.\]
\end{lemma}
We will skip the proof of \lref{deterministicconv} since it is almost the same as \cite[lemma 4.6]{ArmstrongCardaliaguet}.

Now we return to prove \lref{union bd} and \lref{almost linear}.

\begin{proof}[Proof of \lref{almost linear}]
Assume that $s \leq t$.  Call 
\[ n = (\tau+Ls)^{2/3} \ \hbox{ and } \ R =2+ t^{-1}(\tau+Ls)^{2/3} e^{C_1 (\tau+Ls)^{1/3}}\]
 with $C_1$ from \pref{localization} and $\tau$, $L$ from \tref{large scale lip}.  Apply \pref{localization} with 
\[ \hbox{$S^1 = \{x \cdot e \leq t\}$ and $S^2 = \{ m(x)\leq \E[m(te)]-M_R(t)\}$}\]
 in the domain $B_{Rt}(0)$ with the parameters $R$ and $n$ chosen above.  By the definition of $M_R(t)$ we have
 \[ S^2 \subset S^1 \ \hbox{ on } \ B_{Rt}(0).\]
 Then
 \begin{equation}\label{e.contain12}
  \{m(x,S^2) \leq \tau+Ls-Cn\} \subset \{m(x,S^1) \leq \tau+Ls\} \ \hbox{ in } \ B_{Rt - \bar{R}_n(\tau+Ls)}(0)
  \end{equation}
 where 
 \[ \bar{R}_n(\tau+Ls) = n e^{C_1(\tau+Ls)/n} = (\tau+Ls)^{2/3} e^{C_1 (\tau+Ls)^{1/3}}.\] 
   Now 
   \[Rt - \bar{R}_n(\tau+Ls) \geq t+s + (R-2)t - \bar{R}_n(\tau+Ls) \geq t+s\]
   by the definitions of $R$ and $n$ and by \tref{large scale lip}
 \[ m((t+s)e,S^1) \leq \tau + Ls\]
 so the containment \eref{contain12} holds at the point $x = (t+s)e$.  The resulting inequality at $(t+s)e$ is
\[ m((t+s)e,S^1) \leq m((t+s)e,S^2)+Cn = m((t+s)e) - \E[m(te)]+M_R(t)+Cn.\]
Taking expectations, applying \lref{union bd}, and using stationarity
\begin{align*}
 \E[m((t+s)e)] &\geq \E[m((t+s)e,S^1)]+\E[m(te)]-\E[M_R(t)]-Cn \\
 &=\E[m(se)]+\E[m(te)]-\E[M_R(t)]-Cn \\
 & \geq \E[m(se)]+\E[m(te)]-C t^{1/2}\log^{1/2}(Rt)-Cn \\
 & \geq \E[m(se)]+\E[m(te)] - Ct^{1/2}(t^{1/3}+\log t )^{1/2}-Ct^{2/3} \\
 & \geq \E[m(se)]+\E[m(te)] - Ct^{2/3},
 \end{align*}
 we have used stationarity to derive the second line and that $s \leq t$ in the fourth line.   A similar comparison argument gives the upper bound.
\end{proof}
 
 \subsection{Regularity of the asymptotic speed with respect to the normal direction}  In this final section we address the regularity of $\bar{c}(e)$ with respect to varying normal direction $e$.  This regularity plays an important role in studying the homogenization problem \eref{level set hom} with general (non-front like) initial data.  
 
The proof follows \cite[Lemma 4.7]{ArmstrongCardaliaguet}, however we need to replace Lipschitz regularity with large scale Lipschitz regularity.  Note that, although the very weak localization estimate, \pref{localization}, still allows for quantitative sublinearity with a H\"older exponent, the effect on the continuity of $\bar{c}$ is more severe.  We get only a logarithmic modulus of continuity for $\bar{c}$.
 \begin{lemma}\label{l.logcont}
 There is $C(c_{\min},c_{\max}, \|\grad c\|_\infty)$ such that for $e_1,e_2 \in S^{1}$
 \[| \bar{c}(e_1) - \bar{c}(e_2)| \leq C|\log|e_1 - e_2||^{-1}. \]
 \end{lemma}
  \begin{proof}
 We use that $H(e_1)$ and $H(e_2)$ are close in Hausdorff distance in $B_R$, distance $CR|e_1 - e_2|$, for any $R \geq 0$.
 
 Let $R \geq C \geq 1$ to be chosen.  There exist smooth compact convex sets $K^i_R$ with \asref{Sassumption}, diameter bounded by $CR$, $K^i_R \subset H(e_i)$, 
 \[ K^i_R \cap B_R = H(e_i) \cap B_R,\]
 $K^1_R$ is the image of $K^2_R$ under the rotation sending $e_1$ to $e_2$, and
 \[ d_H(K^1_R,K^2_R) \leq CR|e_1-e_2|. \]
 Then by \tref{mainreg} part \partref{setreg}
 \[ |m(x,K^1_R) - m(x,K^2_R)| \leq CR|e_1 - e_2|.\]
 
 Now we fix $s = C_1^{-3}|\log|e_1-e_2||^{3}$, $n = [s^{2/3}]$, and $R = n e^{C_1s/n}$ with $C_1$ from \pref{localization}.  Apply \pref{localization} to see
 \[ m(se,H(e_i))\leq m(se,K^i_R) \leq m(se,H(e_i))+Cn. \]
 Combining all of the above estimates
 \begin{align*}
  |\frac{1}{\bar{c}(e_1)} - \frac{1}{\bar{c}(e_2)}| &\leq Cs^{-1/3} + \frac{1}{s}|\E[m(se,\mathcal{H}^-_{e_1})] - \E[m(se,\mathcal{H}^-_{e_2})]| \\
  & \leq Cs^{-1/3}+C\frac{n}{s}+\frac{1}{s}|m(x,K^1_R) - m(x,K^2_R)| \\
  & \leq Cs^{-1/3}+C\frac{n}{s}+C\frac{1}{s}R|e_1-e_2| \\
  & = C[s^{-1/3}+\frac{n}{s}+\frac{n}{s}e^{C_1s/n}|e_1-e_2|)] \\
  & = C|\log|e_1 - e_2||^{-1}.
  \end{align*}
  Since $\bar{c}(e) \geq \min c >0$ we can also obtain the same continuity estimate for $\bar{c}(e)$. 
 \end{proof}

 \appendix
  \section{Random fluctuations}\label{a.randomfluct}
  
  In this section we prove the fluctuations bound \pref{fluctuations} following \cite{ArmstrongCardaliaguetSouganidis,ArmstrongCardaliaguet} which is based on an idea from first passage percolation by Kesten \cite{Kesten}.  Essentially the idea is to construct a filtration $\{\mathcal{G}_t\}_{t \geq 0}$ of the probability space $(\Omega,\mathcal{F})$ so that $\E[m(x,S)|\mathcal{G}_t]$ has bounded increments in $t$ almost surely allowing us to use Azuma's inequality.  In the discrete i.i.d. setting the $\sigma$-algebra $\mathcal{G}_t$ would just be the smallest $\sigma$-algebra making $t \mapsto S_t = \{m(x,S) \leq t\}$ measurable, because of the continuum setting and the finite range of dependence the definition will be a bit more complicated.  

  \subsection{Localization in sub-level sets}  One of the key observations behind the above idea is the ``localization in sub-level sets" property of $m$, which is just another way of saying that, for a monotonically advancing front, the evolution $S_t$ at time $t$ only depends on the values of $c$ in $S_t$.
\begin{lemma}\label{l.localization in sublevels}
Fix coefficients $c_1, c_2 \in \Omega$.  Suppose that $t \geq 0$ and 
$$ c_1 \equiv c_2 \ \hbox{ in } \ \{ x \in \R^d \setminus S:  \ m(x,S,c_1) \leq t\}.$$
Then it holds that,
$$ m(x,S,c_1) = m(x,S,c_2) \ \hbox{ in } \ \{ x \in \R^d \setminus S:  \ m(x,S,c_1) \leq t\}.$$
\end{lemma}
\begin{proof}
The proof is by a comparison principle.  We call $m_j(\cdot) = m(\cdot,S,c_j)$ and $F_j$ the corresponding PDE operators for $j=1,2$.  We check, using that the equation is geometric and the assumption of the Lemma, that we have,
\[w(x) := m(x,S,c_1) \wedge t \ \hbox{ solves } \ F_2(D^2w,Dw,x) \leq 1 \ \hbox{ in } \R^d \setminus S.\]
Then by comparison principle for the metric problem \lref{localization} $m(x,S,c_1) \wedge t \leq m(x,S,c_2)$.  Thus if $m(x,S,c_2) \leq t$ then so is $m(x,S,c_1)$ and so we can apply the same argument with
\[w(x) := m(x,S,c_2) \wedge t \ \hbox{ solves } \ F_1(D^2w,Dw,x) \leq 1 \ \hbox{ in } \R^d \setminus S.\]
to obtain the other inequality $m(x,S,c_2) \wedge t \leq m(x,S,c_1)$.
\end{proof}

\subsection{Martingale construction}  
\medskip

Since $\{m(x,S) \leq t\}$ is a function on $\R_+$ taking values in the space of compact subsets of $\R^2$ we will need to work in the space,
$$ \mathcal{K} = \textnormal{ the set of $K$ compact in $\R^2$ with $S \subseteq K$}.$$
The space $\mathcal{K}$ comes with the natural Hausdorff metric $d_H$ defined as,
$$ d_H(K,K') = \inf\{ r>0: K \subseteq K' + B_r \ \hbox{ and } \ K' \subseteq K + B_r\}.$$
We remark that $\{m(x,S) \leq t\}$ will be a continuous map $\R_+ \to \mathcal{K}$ under this metric with continuity and monotonicity estimates for $t \geq s \geq 0$,
\begin{equation}
\{m(x,S) \leq s\} + (\tfrac{1}{L}(t-s)-\tau)_+B_1\subseteq\{m(x,S) \leq t\} \subseteq \{m(x,S) \leq s\} + (1+C(t-s))B_1 
\end{equation}
see \lref{sub-level cont}.  

\medskip

In order to make sense of events like $\{\{m(x,S) \leq t\} = K\}$, in analogy to the discrete case where $\{m(x,S) \leq t\}$ can only take finitely values, we introduce a discretization of the space $\mathcal{K}$.  Since the metric space $(\mathcal{K},d_H)$ is locally compact there is a pairwise disjoint partition $(\Gamma_i)_{i \in \mathbb{N}}$ of $\mathcal{K}$ into Borel sets of $\mathcal{K}$ satisfying that $\textup{diam}_H(\Gamma_i) \leq 1$ for every $i \in \mathbb{N}$.  We represent each $\Gamma_i$ by a set $K_i$ which is the closure of $\cup_{K \in \Gamma_i} K+B_{R_0}$, we also define
\[ \widetilde{K}_i = K_i + B_{2}.\]
Note that by the $1$-dependence we have that the $\sigma$-algebras $\mathcal{F}_{K_i}$ and $\mathcal{F}_{\R^2\setminus \widetilde{K}_i}$ are independent.  Now we can view the events,
\[ E_i(t) = \big\{ \{m(x,S) \leq t\} \in \Gamma_i\big\}\]
as being sufficiently fine approximations of the (possibly zero measure) events $\big\{\{m(x,S) \leq t\} = K\big\}$.  For each $t>0$ the $E_i(t)$ are a disjoint partition of $\Omega$,
\[\bigcup_{i \in \mathbb{N}} E_i(t) = \Omega \ \hbox{ and the $E_i(t)$ are pairwise disjoint.}\]
This is a straightforward consequence of $\Gamma_i$ being a partition of $\mathcal{K}$.  The key use of the localization in sub-level sets \lref{localization in sublevels} is to show,
\begin{lemma}\label{l.locality of E_i}
For every $0 < s \leq t$ and $i \in \mathbb{N}$,
$$E_i(t) \in \mathcal{F}_{K_i}.$$
In particular $E_i(t)$ is independent of every set of $\mathcal{F}_{\R^2 \setminus \widetilde{K}_i}$.
\end{lemma}

\begin{proof}
This is a consequence of \lref{localization in sublevels} and the $1$-dependence of the random field $c$.  We define $m^{K_i}(x,S)$, a metric problem solution with localized dependence on the coefficients, as the solution of,
\begin{equation}\label{e.localized m def}
\left\{
\begin{array}{ll}
 -\tr\left[(I - \tfrac{Dm^{K_i} \otimes Dm^{K_i}}{|Dm^{K_i}|^2})D^2m^{K_i}\right] + c^{K_i}(x) |Dm^{K_i}| =  1 &  \hbox{ in } \ \R^d \setminus S \\
m^{K_i} = 0 & \hbox{ on } \ S.
\end{array}\right.
 \end{equation}
With $c^{K_i}$ a $C\|\grad c\|_\infty$ Lipschitz function on $\R^d$ with $c^{K_i}(x) = c(x)$ for $x \in \cup_{K \in \Gamma_i} K$ and $c^{K_i}(x) = c_{\min}$ in the complement of $K_i$. Precisely we define $c^{K_i}$ to be the minimal $C\|\grad c\|_{\infty}$-Lipschitz extension of $c(x)|_{\cup_{K \in \Gamma_i}K}$ to $\R^d$ which is also $\geq c_{\min}$.  For $C$ sufficiently large dimensional constant this exists and equals to $c_{\min}$ on the complement of $K_i$ as claimed.  

Then immediately $m^{K_i} \in \mathcal{F}_{K_i}$ and thus the event $F_i(t) := \{ \{m^{K_i}(x,S) \leq t\} \in \Gamma_i\}$ is $\mathcal{F}_{K_i}$ measurable as well.  On the event $E_i(t)$ we have $\{m(x,S) \leq t\} \subseteq K_i$ and so by \lref{localization in sublevels} it holds that,
\[ m(x,S) = m^{K_i}(x,S) \ \hbox{ for } \ x \in \{m(x,S) \leq t\} \ \hbox{ and thus $F_i(t)$ occurs as well.}\]
A similar argument shows that $F_i(t) \subseteq E_i(t)$ and so $F_i(t) = E_i(t)$ and the result is proven.
\end{proof}
This allows us to define the discretization of $t \mapsto S_t = \{ m(x,S) \leq t\}$,
\begin{equation}
\bar{S}_t := \sum_{i\in \mathbb{N}} K_i {\bf 1}_{E_i(t)}.
\end{equation}
From the setup that $\diam_h(\Gamma_i) \leq 1$ and the definition of $K_i$ we have the discretization error estimate,
\begin{equation}
 S_t \subseteq \bar{S}_t  \subseteq S_t+ B_{1+R_0}.
 \end{equation}
 For the discretized process we have the continuity estimate which follows from \tref{mainreg} part \partref{settimereg},
\begin{equation}\label{e.disc cont}
d_H(\bar{S}_t,\bar{S}_s) \leq  M|t-s|+2R_0+3.
\end{equation}
 Now we define the minimal filtration which makes $\bar{S}_t$ adapted, $\mathcal{G}_0$ is the trivial $\sigma$-algebra and,
 $$ \mathcal{G}_t:= \sigma\bigg(\textnormal{$E_i(s) \cap F$:  $0\leq s \leq t$, $i \in \mathbb{N}$ and $F \in \mathcal{F}(K_i)$}\bigg).$$
Now, as usual for martingale based concentration bounds, for a fixed $x \in \R^2$ we decompose $m(x,S) - \E[m(x,S)]$ by a sum of martingale differences.  Consider the $\mathcal{G}_t$-adapted martingale,
\begin{equation}
X_t := \E[m(x,S)|\mathcal{G}_t]-\E[m(x,S)].
\end{equation}
Since $\mathcal{G}_0$ is trivial we have $X_0 = 0$, while for $t$ sufficiently large, $t \geq \tau+ Ld(x,S)$ we are guaranteed that $m(x,S) \leq t$ and so we expect that $m(x,S)$ is almost (because of the discretization) $\mathcal{G}_t$ measurable. We make this rigorous with the following,
\begin{lemma}\label{l.measurable at t}
For every $0 < s \leq t$ and $x \in \R^2$,
$$ |m(x,S) - \E[m(x,S)|\mathcal{G}_t]|{\bf 1}_{x \in \bar{S}_s} \leq 8\tau$$
and in particular when $t \geq \tau+L d(x,S)$,
$$|m(x,S) - \E[m(x,S)|\mathcal{G}_t]| \leq 8\tau.$$
\end{lemma}
The proof of \lref{measurable at t} will follow at the end of the section.  Now to apply Azuma's inequality we aim to show the bounded differences estimate for $0 < s \leq t$,
 \begin{equation}\label{e.bounded increments}
 |X_t - X_s| \leq C_1 |t-s|+C_2.
 \end{equation}
 With \lref{measurable at t} in mind we make the decomposition,
\begin{align*}
X_t &= \E[m(x,S)|\mathcal{G}_t]{\bf 1}_{x \in \bar{S}_s}+\E[m(x,S){\bf 1}_{x \notin \bar{S}_s}|\mathcal{G}_t]-\E[m(x,S)] \\
X_s &=  \E[m(x,S)|\mathcal{G}_s]{\bf 1}_{x \in \bar{S}_s}+\E[m(x,S){\bf 1}_{x \notin \bar{S}_s}|\mathcal{G}_s]-\E[m(x,S)]
\end{align*}
and see that the first terms are the same modulo constants via \lref{measurable at t} so we can estimate,
$$ |X_t - X_s| \leq |\E[m(x,S){\bf 1}_{x \notin \bar{S}_s}|\mathcal{G}_t]-\E[m(x,S){\bf 1}_{x \notin \bar{S}_s}|\mathcal{G}_s]| + 16\tau $$
Since we are now on the event $x \in \R^2\setminus \bar{S}_s$ we can use the semi-group type property can replace $m(x,S)$ by $m(x,\bar{S}_s)+s$ up to a constant error coming from the discretization,
\begin{lemma}\label{l.semigroup}
For every $t>0$ and $x\in \R^2 \setminus S_t$,
$$|m(x,S)- (t+m(x,\bar{S}_t))| \leq 4\tau$$
\end{lemma}
The proof of \lref{semigroup} is relatively simple and is omitted, the argument is by comparison principle in $\R^2 \setminus \bar{S}_t$ and uses \tref{mainreg} part \partref{setreg} along with the fact (which follows immediately from the definitions) that $\bar{S}_t \subset \{m(x,S) \leq t\} + B_{1+R_0}$.  

\medskip
\lref{semigroup} allows us to rewrite,
$$ |X_t - X_s| \leq |\E[m(x,\bar{S}_s){\bf 1}_{x \notin \bar{S}_s}|\mathcal{G}_t]-\E[m(x,\bar{S}_s){\bf 1}_{x \notin \bar{S}_s}|\mathcal{G}_s]| +20\tau $$
We can also pull out ${\bf 1}_{x \notin \bar{S}_s}$ from both the expectations at this stage since it is $\mathcal{G}_s$ measurable and then bound it by $1$.  Then using the continuity estimate \eref{disc cont} to bound $d_H(\bar{S}_t,\bar{S}_s)$ and the continuity of $m(x,\cdot)$ with respect to $d_H$ from \eref{disc cont},
$$ |X_t - X_s| \leq |\E[m(x,\bar{S}_t)|\mathcal{G}_t]-\E[m(x,\bar{S}_s)|\mathcal{G}_s]| +M |t-s|+24\tau$$
Now finally we are able to use that $\bar{S}_t$ is $\mathcal{G}_t$ measurable (respectively $\bar{S}_s$ is a $\mathcal{G}_s$ measurable) to evaluate the conditional expectations,
\begin{lemma}\label{l.de}
For every $t >0$ and $x \in \R^2$,
\[|\E [m(x,\bar{S}_t)|\mathcal{G}_t] - \sum_{i \in \mathbb{N}} \E[m(x,K_i)]{\bf 1}_{E_i(t)}| \leq \tau+2L.\]
\end{lemma}
The proof of \lref{de} can be found below.  Applying \lref{de} twice we obtain,
$$ |X_t - X_s| \leq \sum_{i,j \in \mathbb{N}}|\E[m(x,K_i)]-\E[m(x,K_j)]|{\bf 1}_{E_i(t)\cap E_j(s)}  +M |t-s|+28\tau.$$
Again using the continuity of $m(x,\cdot)$ with respect to $d_H$ and then the continuity estimate of $\bar{S}_\cdot$ from \eref{disc cont},
\begin{align*}
|X_t - X_s| &\leq \sum_{i,j \in \mathbb{N}}(\tau+ Ld_H(K_i,K_j)){\bf 1}_{E_i(t)\cap E_j(s)}  +M |t-s|+28\tau \\
&=\tau+L d_H(\bar{S}_t,\bar{S}_s)+M |t-s|+28\tau \\
&\leq (1+L)M |t-s|+34\tau.
\end{align*}
This completes the proof of the bounded increments estimate \eref{bounded increments}.

\medskip

Now we are finally able to apply Azuma's inequality. Fix $T =(\tau+L d(x,S)) \vee 1$ so that by \lref{measurable at t}
\begin{equation}\label{e.x_T to m}
|X_T - (m(x,S) - \E m(x,S))| \leq 8\tau .
\end{equation}
Then we divide the interval $[0,T]$ up into $N = T/T_0$ increments of length $T_0$.  Then Azuma's inequality yields for every $\lambda>0$,
$$ \P(|X_T| > \lambda) \leq 2 \exp\left( - \frac{\lambda^2}{2N(2LMT_0+34\tau)^2}\right).$$
Evidently it is best, up to constants, to choose $T_0 = R_0/M$ which results in the estimate,
$$ \P(|X_T| > \lambda) \leq 2\exp\left( - \frac{R_0\lambda^2}{CM\tau^2T}\right).$$
As per \eref{x_T to m} we can replace $X_T$ by the fluctuation $m(x,S) - \E m(x,S)$,
$$ \P(|m(x,S) - \E m(x,S)| > \lambda) \leq C\exp\left( - \frac{c_{\min} \min\{1,c_{\min}^2\}\lambda^2}{Cc_{\max}d(x,S)}\right).$$
This is the fluctuation estimate claimed in \pref{fluctuations}.

\medskip

We complete the proof of \pref{fluctuations} with proofs of \lref{measurable at t} and \lref{de}.

\begin{proof}[Proof of \lref{measurable at t}]
The second part follows from the first and \lref{sub-level cont} since for $t \geq \tau+L d(x,S)$ it holds that $x \in S_t$ almost surely.  We again make use of the localization $m^{K_i}(x,S)$, defined above in \eref{localized m def} which is $\mathcal{F}(K_i)$ measurable. Note that by the definition $m^{K_i}$ satisfies the same estimates as $m$ (e.g. \tref{large scale lip}).  As in the proof of \lref{locality of E_i}, 
\[ m(x,S) = m^{K_i}(x,S) \ \hbox{ on the event $E_i(s)$ and $x \in \cap_{K \in \Gamma_i} K$.}\]
Then by the Lipschitz continuity we have,
\[ |m(x,S)-  m^{K_i}(x,S)|{\bf 1}_{E_i(s)} \leq 2(\tau + LR_0) \ \hbox{ for all $x \in K_i \subset \cap_{K \in \Gamma_i} K+B_{1+R_0}$.}\]
Thus we can replace $m(x,S){\bf 1}_{x \in \bar{S}_s} = \sum_{i : x\in K_i} m(x,S){\bf 1}_{E_i(s)} $ with $\sum_{i : x \in K_i} m^{K_i}(x,S){\bf 1}_{E_i(s)}$ at the cost of a constant error,
\begin{equation}\label{e.est Ks}
 |m(x,S){\bf 1}_{x \in \bar{S}_s}-  \sum_{i: x \in K_i}m^{K_i}(x,S){\bf 1}_{E_i(s)}| \leq 2(\tau + LR_0) = C_0.
 \end{equation}
Making this substitution and using that $m^{K_i}(x,S){\bf 1}_{E_i(s)}$ is $\mathcal{G}_s$ (and hence $\mathcal{G}_t$) measurable we get,
\begin{align*}
|m(x,S) - \E[m(x,S)|\mathcal{G}_t]|{\bf 1}_{x \in \bar{S}_s}&\leq C_0+|\sum_{i: x \in K_i}m^{K_i}(x,S){\bf 1}_{E_i(s)}-\E[m(x,S){\bf 1}_{x \in \bar{S}_s}|\mathcal{G}_t]| \\
&=C_0+|\E[\sum_{i: x \in K_i}m^{K_i}(x,S){\bf 1}_{E_i(s)}-m(x,S){\bf 1}_{x \in \bar{S}_s}|\mathcal{G}_t]| \\
&\leq 2C_0
\end{align*}
where in the last line we have again used \eref{est Ks}.
\end{proof}

\begin{proof}[Proof of \lref{de}]
Here is where the $1$-dependence of the random field is really put to use.  We first claim that for every $i \in \mathbb{N}$ and $t >0$,
\begin{equation}\label{e.ind of}
 \E[m(x,\tilde{K}_i){\bf 1}_{E_i(t)}|\mathcal{G}_t]  = \E[m(x,\tilde{K}_i)]{\bf 1}_{E_i(t)}.
 \end{equation}
To prove this we need to show the following independence, for every $A \in \mathcal{G}_t$,
$$ \E[m(x,\tilde{K}_i){\bf 1}_{E_i(t) \cap A}] =  \E[m(x,\tilde{K}_i)] \P(E \cap A).$$
Note that $m(x,\tilde{K}_i) \in \mathcal{F}(\R^2 \setminus \tilde{K}_i)$, we wish to show that for $A \in \mathcal{G}_t$ we have $A \cap E_i(t) \in \mathcal{F}(K_i)$ then we will be done because of the $1$-dependence.  Since events of the form $A = F \cap E_j(s)$ for $j \in \mathbb{N}$, $F \in \mathcal{F}(K_j)$ and $0 < s \leq t$ generate $\mathcal{G}_t$ we can just argue for such events.   In that case $A \cap E_i(t) = F \cap E_j(s) \cap E_i(t)$.  If $\P(E_j(s) \cap E_i(t)) = 0$ then we are done, otherwise we can show by the monotonicity, that (almost) $K_j \subseteq K_i$.  More precisely, if $E_j(s) \cap E_i(t)$ has positive probability then on that event,
$$ K_j \subseteq \{ m(x,S) \leq s\}+B_{1+R_0} \subseteq \{ m(x,S) \leq t\}+B_{1+R_0} \subseteq K_i + B_1.$$
Of course taken together this is a deterministic containment and so we obtain that $A \cap E_i(t) \in \mathcal{F}(K_i+B_1)$ and hence it is independent of $\mathcal{F}(\tilde K_i)=\mathcal{F}(K_i+B_2)$. This completes the proof of \eref{ind of}.

\medskip

Now we can make use of the independence \eref{ind of} to compute,
\begin{align*}
\E[m(x,\bar{S}_t)|\mathcal{G}_t] &= \sum_{i \in \mathbb{N}} \E[m(x,K_i){\bf 1}_{E_i(t)}|\mathcal{G}_t] \\
&\leq \sum_{i \in \mathbb{N}} \E[m(x,\tilde{K}_i){\bf 1}_{E_i(t)}|\mathcal{G}_t] + \tau+2L \\
&=\sum_{i \in \mathbb{N}} \E[m(x,\tilde{K}_i)|\mathcal{G}_t]{\bf 1}_{E_i(t)} + \tau+2L \\
& \leq \sum_{i \in \mathbb{N}} \E[m(x,K_i)|\mathcal{G}_t]{\bf 1}_{E_i(t)} + \tau+2L.
\end{align*}
The other direction is similar.
\end{proof}

\section{Localized uniqueness}\label{a.localization}
The proof of \lref{localization} somewhat technically involved so we give a brief outline and explanation of the ideas, emphasizing the places where our proof needs to differ from that of \cite{ArmstrongCardaliaguet}.  

At a high level the proof of finite speed of propagation is basically a quantified proof of uniqueness.  Intuitively speaking the idea of the proof is to make a bending in the interface $S^1_t$ to $\tilde{S}^1_t$ so that $\tilde{S}^1_t$ almost solves the same forced mean curvature flow and the two sets are close near the origin, but the complement of $B_R$ is always contained in $\tilde{S}^1_t$ for positive times.  Again speaking intuitively $\tilde{S}^1_t$ is basically an interpolation between $S^1_t$ and a shrinking ball supersolution.  The way of making this bending is basically by looking at an evolution like $S_{t+\varphi(t,x)}$ where $\varphi$ is very slowly varying so that the solution property is only slightly perturbed.  This kind of perturbation is much more straightforward to analyze at the level of the metric problem where it is simply additive.

The main new issues which arise in our case are the following (1) the Lipschitz estimate holds for $m(x) \wedge s$ not $m(x)$ itself, (2) the Lipschitz estimate grows exponentially in $s$ affecting the choice of all the parameters, (3) we need to work always with $m(x,S)$ and not with solutions of the metric problem with general boundary data on a half space.  The final issue is probably the most interesting difference.  Our results \tref{mainreg} parts \partref{sslip} and \partref{largescalelip} rely on connection of the metric problem with the interface evolution, this connection is a bit more difficult to exploit in the case of nontrivial boundary data on $\partial S$ since this amounts to a boundary data problem for \eref{level set} which is of course more difficult to analyze than a problem in the whole space.  This is the reason that \lref{localization} is phrased somewhat differently than \cite[Proposition 4.2]{ArmstrongCardaliaguet}, in the end we believe this phrasing also makes the connection with the underlying interface evolution more clear.

\begin{proof}[Proof of \lref{localization}]
1. We claim it suffices to prove that, under all the assumptions of the Lemma, weaker conclusion holds
\begin{equation}\label{e.wkconc}
 {\bf 1}_{\{m(x,S^2) \leq s-1\}}(0) \leq {\bf 1}_{\{m(x,S^1) \leq s\}}(0).
 \end{equation}
First we prove that this weaker result suffices to prove the full Lemma.  Let $x_0 \in B_{R-\bar{R}(s)}(0)$.  Note that by assumption $R - |x_0| \geq R(s)$ and so the weaker conclusion \eref{wkconc} holds and
\[ {\bf 1}_{\{m(x,S^2) \leq s-1\}}(x_0) \leq {\bf 1}_{\{m(x,S^1) \leq s\}}(x_0).\]
Since $x_0 \in B_{R-\bar{R}(s)}(0)$ was arbitrary we conclude
\[ \{m(x,S^2) \leq s-1\} \subset \{m(x,S^1) \leq s\}   \ \hbox{ on } \ B_{R-\bar{R}(s)}(0).\]
  2. Now we prove the weaker conclusion of the Lemma \eref{wkconc}.  Define,
\[ w^i(x,t) = \min\{ m(x,S^i),t\},\]
which are solutions of the forced mean curvature flow equation,
\[ w^i_t - \tr[(I - \tfrac{Dw^i \otimes Dw^i}{|Dw^i|^2})D^2w^i]+c(x)|Dw^i| = 1\ \hbox{ in $\R^d \setminus S^i \times (0,\infty)$}.\]
We introduce a slight perturbation of $w^1$ which have the effect of localizing the maximum of $w^1(x) - w^2(y)$ near to the origin while only making a small change to the gradient and Hessian of $w^1$.  Fix $\ep>0$ to be chosen in the course of the proof (depending on $s$) and take $g: \R \to [0,\infty)$ to be smooth convex and non-decreasing satisfying,
\[ g(t) = \ep t + \tfrac{1}{3} \ \hbox{ for } \ t \in [0,\infty) \ \hbox{ and } \ \sup_{t \in \R} g'(t) \leq \ep, \ \sup_{t \in \R} g''(t) \leq \ep.\]
Now based on $g$ we define our localizing perturbation,
\[ \psi(x,t) := g((1+|x|^2)^{1/2}-(s-t)).\]
Let $\delta>0$ and $2> \eta >1$ to be chosen later depending on $\ep$ (and hence on $s$) and consider the doubled variable function $\Psi: \bar{S^1} \times \bar{S^2} \times [0,s] \to \R$ defined by,
\[ \Psi(x,y,t) = w^1(x,t) - \eta w^2(y,t) - \frac{|x-y|^4}{4\delta}-\psi(x,t).\]
The function $\Psi$ attains its maximum on $\bar{S^1} \times \bar{S^2} \times [0,s]$ since $w^1$ is bounded, $w^2$ is non-negative and $\psi(x,t) \to +\infty$ as $|x| \to \infty$ uniformly over $t \in [0,s]$.  Call the point $(x_0,y_0,t_0)$ where the maximum of $\Psi$ is attained. 

\medskip

One can immediately obtain an upper bound on $|x_0 - y_0| \lesssim \delta^{1/4}$ using the fact that $w^1$ is bounded from above but this is not sufficient for the comparison argument.  Some quantitative continuity estimate is necessary (since the desired result is quantitative) to obtain a better bound of $|x_0 - y_0|$ and this is why we need the local Lipschitz estimate of \tref{mainreg} part \partref{sslip}.  From the maximizing property we know,
\[  \Psi(x_0,y_0,t_0) \geq \Psi(y_0,y_0,t_0)    \]
and since $w^1$ is Lipschitz continuous in $x$ on $[0,s]$ with constant $\exp(\|Dc\|_\infty s)$,
\[ \frac{|x_0-y_0|^4}{4\delta} \leq w^1(x_0,t_0) - w^1(y_0,t_0) +\psi(y_0,t_0) - \psi(x_0,t_0) \leq (C\ep + e^{\|Dc\|_\infty s}) |x_0 - y_0|,\]
or rearranging,
\begin{equation}\label{e.local lip to x_0y_0}
|x_0 - y_0| \leq Ce^{\|Dc\|_\infty s/3}\delta^{1/3}.
\end{equation}
This estimate will play an important role later in dealing with the $x$-dependence of the coefficients.

\medskip

We claim that, under a good choice of $\eta$, the maximum of $\Psi$ has to be obtained on the parabolic boundary,
\[ \hbox{ either } t_0 = 0 \ \hbox{ or one of } \ x_0 \in \partial S^1 \ \hbox{ or} \  y_0 \in \partial S^2.\]
Suppose otherwise, that is we assume $t_0 >0$ and $(x_0,y_0) \in S^1 \times S^2$ and so $\Psi$ has an interior local maximum at $(x_0,y_0,t_0)$.

\medskip

Now we are able to apply the standard viscosity solution comparison principle results, see \cite{user} Theorem 8.3, to obtain that there are $2\times 2$ symmetric matrices $X$ and $Y$ and $b_1,b_2 \in \R$ with,
\begin{equation*}
\left\{ 
\begin{array}{l}
(X,\xi_0 + D\psi(x_0,t_0), b_1) \in \mathcal{P}^{2,+}_{x,t}w^1(x_0,t_0) \\
(Y, \eta^{-1}\xi_0, \eta^{-1}b_2) \in \mathcal{P}^{2,-}_{x,t}w^2(y_0,t_0)
\end{array}\right. ,
\end{equation*}
with $b_1- b_2 = \partial_t \psi(x_0,t)$ both non-negative, $\xi_0 := \delta^{-1}|x_0 - y_0|^2(x_0 - y_0)$ and the matrices $X,Y$ satisfy for every $\gamma >0$,
\begin{equation}\label{e.comp prop}
-(\tfrac{1}{\gamma}+|A|)I_{4 \times 4} \leq 
\left(\begin{array}{cc}
X + D^2\psi(x_0,t_0) & 0 \\
0 & -\eta Y
\end{array}\right)
 \leq A + \gamma A^2.
\end{equation}
where
\[ A = \frac{1}{\delta} \left(\begin{array}{cc}
B & -B \\
-B & B
\end{array}\right) \ \hbox{ with } \ B = |x_0-y_0|^2 I_{2 \times 2} + (x_0-y_0) \otimes (x_0-y_0). \]
For $x_0 \neq y_0$, which we will show below is the case, we take $\gamma = |A|^{-1}$ to get,
\begin{equation}\label{e.at the max}
-\frac{C}{\delta}|x_0-y_0|^2I_{4 \times 4} \leq 
\left(\begin{array}{cc}
X + D^2\psi(x_0,t_0) & 0 \\
0 & -\eta Y
\end{array}\right)
 \leq \frac{C}{\delta}|x_0-y_0|^2 \left(\begin{array}{cc}
I_{2 \times 2} & -I_{2 \times 2} \\
-I_{2 \times 2} & I_{2 \times 2}
\end{array}\right).
\end{equation}
where the second inequality can be seen by writing out the definition of the matrix inequality and using the bound $ B \leq 3|x_0-y_0|^2I_{2 \times 2}$.  One obtains immediately from \eref{at the max} by testing the left inequality by $(v,v)$ on the left and $(v,v)^t$ on the right that,
\[ (X-\eta Y) \leq -D^2\psi(x_0,t_0) \leq 0,\]
using convexity of $\psi$ in $x$ for the second inequality.

\medskip

Now we show that $x_0 \neq y_0$ in fact with a lower bound on their distance, guaranteeing thereby a lower bound on $|\xi_0|$ so that we can avoid the singularity of the equation at $\xi_0 = 0$.  For this we we use the supersolution property of $w^2$,
\begin{equation}\label{e.2 viscosity cond}
 \eta^{-1}b_2 - \tr^*((I - \tfrac{\xi_0 \otimes \xi_0}{|\xi_0|^2})Y) +c(y_0)|\xi_0| \geq 1.
 \end{equation}
By multiplying the condition at the maximum \eref{at the max} on the left and right by vectors $(0, v)$ and $(0,v)^t$ respectively we have 
\[-Y \leq C(\eta\delta)^{-1}|x_0 - y_0|^2 \leq C\delta^{-1}|x_0 - y_0|^2,\]
if $|\xi_0| = 0$ then use \eref{comp prop} instead with any value of $\gamma$ to get $-Y \leq 0$.  Now plugging this back into the supersolution condition, 
\[ \eta^{-1}b_2+C\delta^{-1}|x_0 - y_0|^2+c(y_0)|\xi_0| \geq 1.\]
Using that $b_1 \leq 1$ and $b_2 = b_1 - \partial_t\psi(x_0,t_0) \leq 1$ this inequality becomes,
\[ C\delta^{-1/3}|\xi_0|^{2/3}+c_{\max}|\xi_0| \geq 1-\eta^{-1} \geq \frac{1}{2}(\eta - 1).\]
From here we see that $|\xi_0| \geq \frac{1}{C}\delta^{1/2}(\eta-1)^{3/2}$.  Given this lower bound we impose the requirement that 
\begin{equation}\label{e.eps req}
\ep \leq \frac{1}{C} \delta^{1/2}(\eta-1)^{3/2}
\end{equation}
 so that, with $\xi_1 = \xi_0 + D\psi(x_0,t_0)$, $|\xi_1| \geq \frac{1}{2}|\xi_0|$.

\medskip

Now we use the lower bound of $\xi_0$ to deal with the perturbation term in the gradient of $w^1$.  From the sub-solution property of $w^1$,
\[ b_1 - \tr\left[(I - \tfrac{\xi_1 \otimes \xi_1}{|\xi_1|^2})(X+D^2\psi(x_0,t_0))\right] +c(x_0)|\xi_1| \leq 1.\]
Now using this equation as well as the lower bound on $\xi_1$ and $\xi_0$ and the upper bounds of $D^2\psi$ and $D\psi$,
\begin{equation}\label{e.1 viscosity cond}
\begin{array}{ll}
b_1 - \tr\left[(I - \tfrac{\xi_0 \otimes \xi_0}{|\xi_0|^2})X\right] +c(x_0)|\xi_0| &\leq 1 + C|\xi_0|^{-1}|D\psi(x_0,t_0)||X|+|D^2\psi(x_0,t_0)|+c_{\max}|D\psi(x_0,t_0)| \vspace{1.5mm}\\
& \leq 1+Ce^{2\|Dc\|_\infty s/3}(\eta-1)^{-3/2}\delta^{-5/6}\ep,
\end{array}
\end{equation}
where we have used \eref{at the max} again to obtain that,
\[ |X| \leq \frac{C}{\delta}|x_0-y_0|^2+C\ep \leq Ce^{2\|Dc\|_\infty s/3}\delta^{-1/3}.\]
Now finally we compare the viscosity solution conditions for $w^1$ and $w^2$ to get a contradiction.  Subtracting \eref{2 viscosity cond} from \eref{1 viscosity cond} and using that $X-\eta Y \leq 0$,
\[ b_1-b_2 + (c(x_0)-c(y_0))|\xi_0| \leq 1-\eta + Ce^{2\|Dc\|_\infty s/3}(\eta-1)^{-3/2}\delta^{-5/6}\ep,\]
which can be rearranging, using $|c(x_0)-c(y_0)||\xi_0| \leq \|Dc\|_\infty\delta^{-1}|x_0-y_0|^4$ and $b_1-b_2 = \partial_t\psi$, to become
\[ \eta-1 \leq Ce^{2\|Dc\|_\infty s/3}\left[(\eta-1)^{-3/2}\delta^{-5/6}\ep + \delta^{1/3}\right] \leq C (\eta-1)^{-3/2}e^{2\|Dc\|_\infty s/3}\left[\delta^{-5/6}\ep + \delta^{1/3}\right] \]
to make the two terms on the right of the same size we choose $\delta := \ep^{6/7}(\eta-1)^{-9/7}$ so that,
\[ (\eta-1)^{10/7} \leq C e^{2\|Dc\|_\infty s/3}\ep^{2/7}\]
which will result in a contradiction if $\eta$ is chosen as,
\[ \eta := 1+Ce^{7\|Dc\|_\infty s/15}\ep^{1/5}.\]
We remark that is choice is indeed consistent with the requirement \eref{eps req} on $\ep$ since given our choices \eref{eps req} becomes $\ep \leq Ce^{Cs}\ep^{33/70}$ which is certainly satisfied when $\ep <1$ and that under this choice $\delta(\ep) \to 0$ as $\ep \to 0$.

\medskip

Now we have justified, given the choices of $\delta, \eta$ as above, that the maximum of $\Psi(x,y,t)$ over $\bar{S^1}\times \bar{S^2} \times [0,s]$ is achieved on the parabolic boundary.  There are still several cases that we need to consider, $(i)$ that $t_0 =0$, $(ii)$ that $x_0 \in \partial S^1$, $(iii)$ that $y_0 \in \partial S^2$.  We will show that in all three cases (at worst),
\begin{equation}\label{e.psi est}
 \max \Psi \leq \tfrac{1}{3}.
 \end{equation}

\medskip

If $t_0 = 0$, both $w^i(x,0) = 0$, so
\[  \max \Psi = w^1(x_0,0)-\eta w^2(y_0,0) -\frac{|x_0-y_0|^4}{4\delta}- \psi(x_0,0) \leq 0.\]
Next we argue when $x_0 \in \partial S^1$, the other case when $y_0 \in \partial S^2$ is similar.  First let us suppose that $x_0 \in B_R$, in that case we use the boundary condition $w^1(x_0,t_0) = 0 \leq w^2(x_0,t_0)$, 
\begin{equation}\notag
\begin{array}{ll}
 \max \Psi &= w^1(x_0,t_0)-\eta w^2(y_0,t_0) -\frac{|x_0-y_0|^4}{4\delta}- \psi(x_0,t_0) \vspace{1.5mm}\\
 & \leq w^2(x_0,t_0)-w^2(y_0,t_0) \vspace{1.5mm}\\
 & \leq \tfrac{2}{c_{\min}}e^{\|Dc\|_\infty s} |x_0-y_0| \leq Ce^{4\|Dc\|_\infty s/3}\delta^{1/3}
 \end{array}
 \end{equation}
 where we have used the local Lipschitz estimate \tref{mainreg} part \partref{sslip} for the second inequality.  For $\ep$ (and hence $\delta$) small enough, in this case $\ep \leq e^{-Cs}$ for a large enough $C$ will suffice and $\ep$ will be chosen as such below, we obtain,
 \[ \max \Psi \leq \tfrac{1}{3}.\]
 Note that here we could have used the large scale Lipschitz estimate instead, but we will need to choose $\ep$ and $\delta$ exponentially small in $s$ anyway in the second paragraph below.
 
 If, on the other hand, $x_0 \in \partial S^1 \setminus B_R$ or $y_0 \in \partial S^2 \setminus B_R$ then, 
 \[ \max \Psi \leq s-\eta w^2(y_0,t) - \psi(x_0,t_0) \leq s - g(R-s) \leq s - \ep (R-s) \leq 0 \]
as long as 
\begin{equation}\label{e.R choice}
R \geq s(1+\frac{1}{\ep}).
\end{equation}
This completes the proof of \eref{psi est}.

\medskip

Now from \eref{psi est} we derive the result.  By the assumed bounds,
\[ w^1(0,s) \leq \eta w^2(0,s) + \frac{1}{3} +\psi(0,s) \leq w^2(0,s) +(\eta-1)s+\frac{2}{3}  \leq w^2(0,s)+1 \]
as long as we choose $\ep$ so that,
\[ (\eta -1)s = Ce^{7\|Dc\|_\infty s/15}\ep^{1/5}s \leq \tfrac{1}{3}\]
i.e.
\[ \ep = e^{-Cs} \ \hbox{ with constant } \ C = C(\|Dc\|_\infty,c_{\min}).\]
Now we can satisfy \eref{R choice} by the constraint $R \geq  e^{Cs}$ and obtain
\[ \min\{m^1(0),s\} \leq \min\{m^2(0),s\} + 1.\]
Now if $m(0,S^2) < s-1$ then $\min\{m(0,S^1),s\} < s-1+1 = s$ and hence $m(0,S^1) < s$.  Thus we obtain \eref{wkconc}.

\end{proof}

\bibliographystyle{plain}
\bibliography{mcf_articles.bib}
\end{document}